\documentclass[11pt]{amsart}

\usepackage[T1]{fontenc}
\usepackage[utf8]{inputenc}
\usepackage[english]{babel}
\usepackage{mathtools}
\usepackage{amsfonts}
\usepackage{amsthm}
\usepackage{comment}
\usepackage{amsmath}
\usepackage{xcolor}
\usepackage{cite}
\usepackage{braket}
\usepackage{amssymb}

\numberwithin{equation}{section}
\theoremstyle{plain}
\newtheorem{theo}{Theorem}[section]

\newtheorem{lemma}[theo]{Lemma}
\newtheorem{coro}[theo]{Corollary}

\newtheorem{question}{Question}
\theoremstyle{definition}
\newtheorem{rem}[theo]{Remark}
\newtheorem{example}[theo]{Example}

\newcommand{\N}{\mathbb{N}}
\newcommand{\C}{\mathbb{C}}
\newcommand{\D}{\mathbb{D}}
\newcommand{\T}{\mathbb{T}}
\newcommand{\B}{\mathcal{B}}
\newcommand{\dyad}{\mathcal{D}}

\DeclareMathOperator{\mre}{Re}

\DeclareMathOperator*{\BMO}{BMO}
\DeclareMathOperator*{\BMOA}{BMOA}

\def\Xint#1{\mathchoice
{\XXint\displaystyle\textstyle{#1}}%
{\XXint\textstyle\scriptstyle{#1}}%
{\XXint\scriptstyle\scriptscriptstyle{#1}}%
{\XXint\scriptscriptstyle\scriptscriptstyle{#1}}%
\!\int}
\def\XXint#1#2#3{{\setbox0=\hbox{$#1{#2#3}{\int}$ }
\vcenter{\hbox{$#2#3$ }}\kern-.6\wd0}}
\def\dashint{\Xint-}

\usepackage{hyperref}

\usepackage{tikz}
\usepackage{ifthen}
\usepackage{xfp}
\usetikzlibrary{calc}

\title{Restrictions of B\'ekollé--Bonami weights and Bloch functions}

\author[A. Dayan]{Alberto Dayan}
\address{Departament de Matemàtiques, Universitat Autònoma de Barcelona, 08193, Barcelona, Spain}
\email{alberto.dayan@uab.cat}

\author[A. Llinares]{Adrián Llinares}
\address{Department of Mathematics and Mathematical Statistics, Ume{\aa} University, SE-90736 Ume{\aa}, Sweden}
\curraddr{Departamento de Matemáticas, Universidad Autónoma de Madrid, 28049 Madrid, Spain}
\email{adrian.llinares@uam.es}

\author[K-M. Perfekt]{Karl-Mikael Perfekt}
\address{Department of Mathematical Sciences, Norwegian University of Science and Technology, 7491 Trondheim, Norway}
\email{karl-mikael.perfekt@ntnu.no}

\date{\today}

\keywords{Békollé--Bonami weights, Rubio de Francia factorization, Bloch space, interpolating sequences}

\subjclass[2020]{46E30 (Primary); 30E05, 30H30, 42A61, 47B38 (Secondary)}

\begin{document}

    \begin{abstract}
        We characterize the restrictions of Békollé--Bonami weights of bounded hyperbolic oscillation to subsets of the unit disc, thus proving an analogue of Wolff's restriction theorem for Muckenhoupt weights. Sundberg proved a discrete version of Wolff's original theorem, by characterizing the trace of $\BMO$-functions onto interpolating sequences. We consider an analogous question in our setting, by studying the trace of Bloch functions. Through Makarov's probabilistic approach to the Bloch space, our question can be recast as a restriction problem for Bloch dyadic martingales on the unit circle.
    \end{abstract}

    \maketitle

    \section{Introduction}
    
        For $1 < p < \infty$, a positive measurable weight $w$ on the unit circle $\mathbb{T}$ is said to be a Muckenhoupt $A_p$-weight if 
            \[
                [w]_{A_p} := \sup_{I \subset \mathbb{T}} \; \dashint_I  w \left( \dashint_I  w^{-\frac{1}{p-1}} \right)^{p-1} < \infty,
            \]
            where the supremum is taken over all subarcs $I$ of the unit circle, $|I|$ denotes the normalized arc length of $I$ and
            \[
                \dashint_I w :=\frac{1}{|I|}\int_{I} w(\theta) \frac{d\theta}{2\pi}
            \]
            denotes the average of $w$ on it. When $p=1$, we say that $w \in A_1$ if there is a minimal constant $C = [w]_{A_1}$ such that $\mathcal{M}w(\theta) \leq Cw(\theta)$ for a.e. $\theta \in \mathbb{T}$, where $\mathcal{M}w$ is the Hardy--Littlewood maximal function of $w$, 
            \[
                \mathcal{M}w(\theta) = \sup_{I \ni \theta} \; \dashint_I w.
            \]
            Muckenhoupt weights are extremely important, because, for $1 < p < \infty$, it holds that $w \in A_p$ if and only if the maximal function $\mathcal{M}$ is bounded on $L^p(\mathbb{T}, w\, d\theta)$, and, furthermore, if and only if the Hilbert transform is bounded on $L^p(\mathbb{T}, w\, d\theta)$. In a famous unpublished preprint, circa 1980, Wolff characterized the restrictions of $A_p$-weights to subsets of the unit circle.
    
        \begin{theo}[Wolff, as cited in \cite{Garcia-CuervaRubiodeFrancia}] \label{thm:wolff}
        	Let $w \geq 0$ be a weight on a measurable subset $\Omega \subset \mathbb{T}$, and let $1 \leq p < \infty$.  Then the following are equivalent.
        	            \begin{enumerate}
        	            			\item There exists $W \in A_p$ such that $w(\theta) = W(\theta)$ for almost every $\theta \in \Omega$.
        		\item There exists $q > 1$ such that $w^q \in A_{p, \Omega}$. That is, if $p > 1$, it holds that
        		\[\sup_{I \subset \mathbb{T}} \left( \frac{1}{|I|}\int_{I \cap \Omega}  w^q \right) \left( \frac{1}{|I|}\int_{I \cap \Omega} w^{-\frac{q}{p-1}} \right)^{p-1} < \infty,\]
        	   and similarly for $p=1$.
        	\end{enumerate}
        \end{theo}
    
        By taking logarithms in the Wolff restriction theorem, one obtains a characterization of the restrictions of $\BMO(\mathbb{T})$-functions to $\Omega$ \cite[Chapter IV, Corollary~5.7]{Garcia-CuervaRubiodeFrancia}. Using this, Sundberg gave in \cite{sundberg} a discrete analogue of Theorem \ref{thm:wolff}, by describing the trace of $\BMOA$ functions on interpolating sequences. We recall that a sequence $Z=(z_n)_{n\in\N}$ in the unit disc is interpolating for the space $H^\infty$ of bounded analytic functions on $\D$ if, given any bounded sequence $(w_n)_{n\in\N}$, there exists a $\phi$ in $H^\infty$ such that $\phi(z_n)=w_n$, $n \in \N$. Let
            \[
                \rho(z, w):=\left|\frac{w-z}{1-\overline{w}z}\right|, \qquad z, w\in\D,
            \]
            be the pseudo-hyperbolic distance on $\D$. Carleson  \cite{Carleson} famously showed  that $Z$ is interpolating if and only if it is weakly separated, $\inf_{n\ne j}\rho(z_n, z_j)>0$,
            and the measure
            \[
                \mu_Z:=\sum_{n\in\N}(1-|z_n|^2)\delta_{z_n}
            \]
            satisfies the Carleson condition: $\mu_Z \big( S(I) \big)\leq C_Z |I|$, where $C_Z > 0$ is the Carleson constant for the measure and $I$ is any arc of the unit circle.  We remark that $Z$ satisfies this latter Carleson condition if and only if 
            \begin{equation}
            	\label{eqn:cm}
            	\sup_{z\in\D}\sum_{n\in\N} \big (1-\rho^2(z, z_n) \big)<\infty.
            \end{equation}
            Moreover, Carleson showed that the Carleson condition is equivalent to the embedding condition
            \begin{equation}
            	\label{eqn:cmPoisson}
            	\sum_{n\in\N}|g(z_n)|^2(1-|z_n|^2)\lesssim_{ Z}\dfrac{1}{2\pi} \int_0^{2\pi}|g(e^{i\theta})|^2\,d\theta, \qquad g\in L^2(\T),
            \end{equation}
            where $g(z)$ denotes the value at $z \in \mathbb{D}$ of the Poisson extension of $g$ to the unit disc.
    
        We now give a slightly modified version of the original statement of Sundberg's result, which can be read from its proof.
        
        \begin{theo}[Sundberg \cite{sundberg}]
        \label{theo:sundberg}
            Given a sequence $(w_n)_{n\in\N}$ in $\C$ and a sequence $Z=(z_n)_{n\in\N}$ in $\mathbb{D}$ satisfying the Carleson condition \eqref{eqn:cm}, there exists a function $f$ in $\BMO(\mathbb{T})$ such that $(f(z_n)-w_n)_{n\in\N} \in \ell^\infty$ if and only if there exists a function $\beta\colon\D\to\C$ and a number $\lambda > 0$ such that
                \begin{equation}
                \label{eqn:sundberg}
                    \sup_{z\in\D}\sum_{n\in\N}e^{\lambda|w_n-\beta(z)|} \big(1-\rho^2(z, z_n) \big)<\infty.
                \end{equation}
                Moreover, if $Z$ additionally  is weakly separated (and hence interpolating), then, for every sequence $(w_n)_{n\in\N}$ satisfying \eqref{eqn:sundberg}, there exists a function $f$ in $\BMO(\mathbb{T})$ with holomorphic Poisson extension such that $f(z_n)=w_n$, for every $n$.
        \end{theo}
    
        The goal of this note is to describe to which extent Theorem \ref{thm:wolff} and Theorem \ref{theo:sundberg} have analog statements in the setting of Békollé-Bonami weights and Bloch functions, respectively, on the unit disc. Following B\'ekoll\'e and Bonami \cite{BekolleBonami78}, for $1 < p < \infty$,  the $B_p$-class is defined as the family of all weights $w$ in the unit disc such $\mathbb{D}$ such that
            \[
            	[w]_{B_p} := \sup_{I \subset \T} ~ \dashint_{S(I)} w \, dA \left( \dashint_{S(I)} w^{-\frac{1}{p-1}} \, dA \right)^{p-1} < \infty,
            \]
            where $dA$ denotes normalized area measure on $\mathbb{D}$, and for an arc $I$, $S(I)$ denotes the corresponding Carleson square, that is, the sector 
            \[
                S(I):=\left\{z\in\D\, : \, \frac{z}{|z|}\in I, 1-|z|<|I|\right\}.
            \]
            For $p=1$, we say that $w \in B_1$ if there is a constant $C$ such that $Mw(z) \leq Cw(z)$ for a.e. $z \in \mathbb{D}$, where $M$ now denotes the associated maximal function,
            \[
                M w (z):=  \sup_{S(I) \ni z} \dashint_{S(I)} w \, dA.
            \]
            As for Muckenhoupt weights, for $1 < p < \infty$, we have that $w \in B_p$ if and only if $M$ is bounded on $L^p(\mathbb{D}, w \, dA)$, and, furthermore, if and only if the Bergman projection is bounded on this weighted $L^p$-space.
    
        However, the analogy stops at a certain point, since B\'ekoll\'e--Bonami weights can be much more irregular than Muckenhoupt weights. In particular, they can easily fail to satisfy the reverse H\"older and self-improvement properties that $A_p$-weights enjoy (cf. Theorem~\ref{thm:wolff} with $\Omega = \mathbb{T}$).   Recently, it was discovered that these desirable properties can be recovered \cite{AlemanPottReguera} if we consider weights $w \in B_p$ which are of \emph{bounded hyperbolic oscillation}, that is, weights for which there is a constant $L_w > 0 $ such that
            \begin{equation} \label{eq:bddhyposc}
                w(z)\le L_w w(\lambda)
            \end{equation}
            provided that $\lambda$ and $z$ belong to the same \emph{top-half} $T(I)$ of a Carleson box:
            \[
                T(I):=\{z\in S(I)\, : \, 1-|z|>|I|/2\}.
            \]
            The least constant $L_w$ for which \eqref{eq:bddhyposc} holds uniformly on all intervals $I$ is called the \emph{constant of bounded hyperbolic oscillation} of the weight $w$. B\'ekoll\'e--Bonami weights satisfying \eqref{eq:bddhyposc}, as well as their connection with the Bloch space, were further explored in \cite{arturadem}.
                    
        Under the regularity assumption \eqref{eq:bddhyposc}, the main theorem of this article establishes the analogue of the Wolff restriction theorem for general subsets $\Omega \subset \mathbb{D}$. For a non-negative weight $w$ on a measurable subset $\Omega \subset \D$, we say that $w \in B_{p, \Omega}$ if
            \begin{equation} \label{RestrictedB_p}
                [w]_{B_{p,\Omega}} := \sup_{I \subset \T} \dfrac{1}{A \big( S(I) \big)}\int_{S(I) \cap \Omega} w \: dA \left( \dfrac{1}{A \big( S(I) \big)} \int_{S(I) \cap \Omega} w^{-\frac{1}{p - 1}} \: dA \right )^{p - 1} < \infty.
            \end{equation}
            When $p = 1$, we define $B_{1, \Omega}$ as the space of all weights $w$ in $\Omega$ such that
            \begin{equation} \label{RestrictedMaximalFunction}
              M_\Omega w (z):=  \sup_{S(I) \ni z} \dfrac{1}{A \big( S(I) \big)} \int_{S(I) \cap \Omega} w \, dA \leq C w(z)
            \end{equation}
            for almost every $z \in \Omega$. We write $[w]_{B_1,\Omega}$ to denote the infimum of the constants $C$ in the above inequality.
    
        Note that if $w \in B_{p, \Omega}$, then $w^{-\frac{1}{p-1}} \in B_{\frac{p}{p-1}, \Omega}$ and $\left[ w^{-\frac{1}{p-1}}\right]_{B_\frac{p}{p-1}, \Omega}$ coincides with $\left [ w \right]_{B_p, \Omega}^\frac{1}{p-1}$.
                
         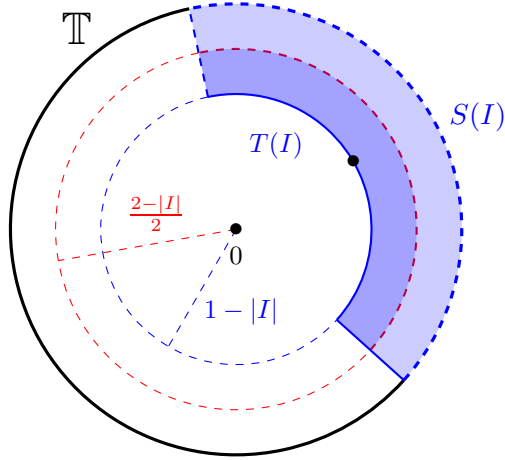
\begin{figure}[h!]
        	\centering
        	\begin{tikzpicture}[x=3cm,y=3cm]
        		\def \mod {0.6};
        		\def \arg {30};
        		\pgfmathsetmacro{\avg}{(1+\mod)/2}
        		\pgfmathsetmacro{\angle}{180*(1-\mod)}
        		
        		\draw[very thick, black] ({\arg + \angle}:1) arc({\arg + \angle}:{\arg - \angle + 360}:1);
        		\draw[very thick, dashed, blue] ({\arg - \angle)}:1) arc({\arg - \angle}:{\arg + \angle}:1);
        		\draw[dashed, blue] ({\arg + \angle}:\mod) arc({\arg + \angle}:{\arg - \angle + 360}:\mod);
        		\draw[thick, blue] ({\arg - \angle)}:\mod) arc({\arg - \angle}:{\arg + \angle}:\mod);
        		\draw[dashed, red] ({\arg + \angle}:{\avg}) arc({\arg + \angle}:{\arg - \angle +360}:{\avg});
        		\draw[thick, dashed, red] ({\arg - \angle}:{\avg}) arc({\arg - \angle}:{\arg + \angle}:{\avg});
        		\draw[dashed, red] (0,0) -- ($ \avg*({cos(\arg + 160)},{sin(\arg + 160))}) $) node[pos=0.45, above=1pt]{\small $\frac{2-|I|}{2}$};
        		\draw[dashed, blue] (0,0) -- ($ \mod*({cos(\arg + 210)},{sin(\arg + 210))}) $) node[pos=0.7, right=2pt]{\small $1-|I|$};
        		\draw (0,0) node{$\bullet$}node[below=3pt]{\small 0};
        		\draw ($ ({cos(135)},{sin(135)}) $) node[above=6pt]{\LARGE $\mathbb{T}$};
        		\draw[thick, dashed, blue] ($ \mod*({cos(\arg + \angle)},{sin(\arg + \angle)}) $) -- ($ \avg*({cos(\arg + \angle)},{sin(\arg + \angle)}) $);
        		\draw[thick, dashed, blue] ($ \avg*({cos(\arg + \angle)},{sin(\arg + \angle)}) $) -- ($ ({cos(\arg + \angle)},{sin(\arg + \angle)}) $);
        		\draw[thick, blue] ($ \mod*({cos(\arg - \angle)},{sin(\arg - \angle)}) $) -- ($ \avg*({cos(\arg - \angle)},{sin(\arg - \angle)}) $);
        		\draw[thick, blue] ($ \avg*({cos(\arg - \angle)},{sin(\arg - \angle)}) $) -- ($ ({cos(\arg - \angle)},{sin(\arg - \angle)}) $);
        		\fill[blue, opacity=0.2] ({\arg - \angle}:\mod) arc({\arg - \angle}:{\arg + \angle}:\mod)--({\arg + \angle}:{\avg}) arc({\arg + \angle}:{\arg - \angle}:{\avg})--cycle;
        		\fill[blue, opacity=0.2] ({\arg - \angle}:\mod) arc({\arg - \angle}:{\arg + \angle}:\mod)--({\arg + \angle}:{1}) arc({\arg + \angle}:{\arg - \angle}:{1})--cycle;
        		\draw[blue] ($ \mod*({cos(\arg + 60*(1-\mod))},{sin(\arg + 60*(1-\mod))}) $) node[below left=1pt]{\small $T(I)$};
        		\draw[blue] ($ ({cos(\arg)},{sin(\arg)}) $) node[right=3pt]{$S(I)$};
                    \draw[black] ($ \mod*({cos(\arg)},{sin(\arg)}) $) node{$\bullet$};
        	\end{tikzpicture}
        	\caption{Example of $S(I)$, $T(I)$ and their centre.}
        	\label{fig:TopHalf}
        \end{figure}
    
        Our first main result is the following:
        
        \begin{theo} \label{theo:main1}
            Let $\Omega \subset \D$ of positive measure, and consider a weight $w$ of bounded hyperbolic oscillation on $\Omega$, that is, a measurable function $w \colon \Omega \to (0,\infty)$ satisfying \eqref{eq:bddhyposc} for points $z, \zeta \in \Omega$. Then, for $1 \leq p < \infty$, the following statements are equivalent:
            \begin{enumerate}
                \item \label{Wolff1)} There exists a weight $W \in B_p$ of bounded hyperbolic oscillation such that $w(z) = W(z)$ for almost every $z \in \Omega$.
                \item \label{Wolff2)} There exists $q > 1$ such that $w^q \in B_{p, \Omega}$. 
            \end{enumerate}
        \end{theo}
     
        The reader familiar with the theory of $A_p$ weights notices that any $B_p$ weight of bounded hyperbolic oscillation on $\D$ is an $A_p$ weight, where the $A_{p}$ characteristic is defined via the basis of all discs with centers in $\D$, intersected with $\D$ (see Lemma \ref{lemma:BpAp}). On the other hand, not every $A_p$ weight of $\D$ is of bounded hyperbolic oscillation (e.g., \cite[Example~7.1.7]{Grafakos}) and neither is every $B_p$ weight of bounded hyperbolic oscillation (see \cite{Borichev}). Hence Theorem \ref{theo:main1} does not quite follow from the known generalizations of Theorem \ref{thm:wolff} for $A_p$ weights on more general domains (see for instance \cite{MR4340793}), since we do not know a priori that the $A_p$ extension is of bounded hyperbolic oscillation. 
    
        The proof of Theorem \ref{theo:main1} is contained in Section \ref{sec:proof}, and it relies on factorization techniques. More specifically, in Section \ref{sec:factorization}  we prove that any $w \in B_p$ of bounded hyperbolic oscillation, $p > 1$, can be factored as $w = w_1w_2^{1-p}$, where the weights $w_1, w_2 \in B_1$ are also of bounded hyperbolic oscillation. This type of factorization problem within a subclass of weights has previously been considered by Borichev \cite[Remark~2]{Borichev}, where he studies a different class of well-behaved B\'ekoll\'e--Bonami weights. Section \ref{sec:BHO} contains the properties of weights of bounded hyperbolic oscillation that we will need for our task. 
    
        With Theorem \ref{theo:main1} at our disposal, it is natural to ask what an analog of Theorem \ref{theo:sundberg} should look like. As mentioned earlier in this Introduction, by taking logarithms in the Wolff restriction theorem, one obtains a characterization of the restrictions of $\BMO(\mathbb{T})$-functions to $\Omega$ \cite[Chapter IV, Corollary~5.7]{Garcia-CuervaRubiodeFrancia}. The same is in principle true in our setting, when considering the space
            \[
                \BMO(\D) := \left \{ f \in L^1_{\textrm{loc}} (\D) \, : \, \sup_D \dashint_{D \cap \D} \left| f - \dashint_{D \cap \D} f \right| \, dA < \infty \right\},
            \]
            where the supremum is taken over all discs $D$ centered at points in $\D$. However, as shown in \cite[Section~3]{arturadem}, the logarithm of any weight of bounded hyperbolic oscillation is automatically in $\BMO(\mathbb{D})$, and the restriction problem therefore boils down to the extension of quasi-Lipschitz functions.
    
        To obtain a more interesting restriction problem, we consider the space $\mathcal{B}$ of \emph{holomorphic} functions in $\BMO(\mathbb{D})$. This space is precisely the Bloch space \cite[Section~5]{CoifmanRochbergWeiss76}, which is usually introduced as the space of holomorphic functions $b$ on the unit disc such that
            \[
                \|b\|_\B:=|b(0)|+\sup_{z\in\D}|b'(z)|(1-|z|^2)<\infty.
            \]
            In other words, $\mathcal{B}$ is the space of holomorphic functions in $\mathbb{D}$ which are  Lipschitz with respect to the hyperbolic metric.

        Given an interpolating sequence $Z$, we therefore seek to characterize the trace of the Bloch space $\mathcal{B}$ onto $Z$. Note that condition \eqref{eqn:sundberg} is directly related to the John--Nirenberg inequality for $\BMO(\mathbb{T})$. For the larger Bloch space we instead have the estimate
            \begin{equation}
            \label{eqn:Hedenmalm}
                \sup_{r\in(0, 1)}\int_0^{2\pi}e^\frac{a~|b(re^{i\theta})|^2}{-\log(1-r^2)}\,\frac{d\theta}{2\pi}\leq C_a,\qquad\|b\|_\B\leq1, \; 0 < a < 1,
            \end{equation}
            if $b(0) = 0$; see \cite{MR761804}, \cite{MR794117} or \cite[Th.~8.9]{Pommerenke}. Other variants and further refinements of \eqref{eqn:Hedenmalm} are presented in \cite{heden}.
    
        The second main result of our paper deals with necessary conditions for a sequence to be the trace of a Bloch function on an interpolating sequence:
        
        \begin{theo}
        \label{theo:necessary}
            Let $b$ be a Bloch function, and let $(z_n)_{n\in\N}$ be a sequence in $\D$ satisfying the Carleson condition \eqref{eqn:cm}. Then
            \begin{description}
            	\item[(i)] There exists $\lambda > 0$ such that
                	\begin{equation}
                		\label{eqn:traceblochhede}
                		\sup_{z\in\D}\sup_{r\in(0, 1)}\sum_{\rho(z, z_n)<r}e^\frac{\lambda|b(z_n)-b(z)|^2}{-\log(1-r^2)} \big(1-\rho^2(z, z_n) \big)<\infty.
                	\end{equation}
                    
                \item[(ii)] There exist $\lambda > 0$ and $C > 0$ such that, for all $z$ in $\D$, the function
                    \[
                        \N \ni n \mapsto e^\frac{\lambda|b(z_n)-b(z)|^2}{-\log(1-\rho^2(z, z_n))} \big (1-\rho^2(z, z_n) \big)
                    \]
                    has weak-$L^1$-norm less than $C$ with respect to the counting measure on $\N$, that is,
                    \begin{equation}
                    \label{eqn:tracebloch}
                        \#\left\{n\in\N \,\bigg|\,e^\frac{\lambda|b(z_n)-b(z)|^2}{-\log(1-\rho^2(z, z_n))} \big (1-\rho^2(z, z_n) \big )>t \right\}\leq \frac{C}{t}, \qquad t>0.
                    \end{equation}
            \end{description}
            Furthermore,
            \begin{description}
                \item [(iii)] There exists a function $b\in \B$ and an interpolating sequence $(z_n)_{n\in\N}$, such that, for every $\lambda>0$, the strong summability condition associated to \eqref{eqn:tracebloch} fails at $z=0$, that is,
                \begin{equation}
                \label{eqn:counter}
                    \sum_{n\in\N}e^\frac{\lambda|b(z_n)-b(0)|^2}{-\log(1-|z_n|^2)}(1-|z_n|^2)=\infty.
                \end{equation}
            \end{description}
        \end{theo}
        
        \begin{rem}
        	Since Bloch functions are Lipschitz, we have that  $|b(z)-b(0)|\lesssim\log\frac{1}{1-|z|^2}.$
        	Therefore \eqref{eqn:counter} also implies that \eqref{eqn:sundberg} fails.
        \end{rem}
    
        Of course, Theorem \ref{theo:necessary} begs the following question:
        
        \begin{question}
        \label{q:Bloch}
            Suppose additionally that $Z = (z_n)_{n \in \N}$ is weakly separated, and hence interpolating. Given a sequence $(w_n)$ in $\C$, are either of the conditions \eqref{eqn:traceblochhede} or \eqref{eqn:tracebloch} (with $w_n$ replacing $b(z_n)$, and a general given function $\beta(z)$ replacing $b(z)$) sufficient for the existence of a Bloch function $b$ such that $b(z_n)=w_n$?
        \end{question}
    
        The proof of Theorem \ref{theo:necessary} is found in Section \ref{sec:martingales}. The main ingredient comes from the insight of Makarov, \cite{makarov}, who showed that Bloch functions are in one-to-one correspondence with the set of dyadic martingales on the unit circle that satisfies a certain Lipschitz condition once seen as functions on the dyadic tree (see \eqref{eqn:blochmartingale} for a precise definition). In light of this correspondence, such dyadic martingale are referred to as \emph{Bloch martingales}. This establishes a probabilistic viewpoint on Bloch functions, which comes with an additional set of tools from probability theory to study properties of the Bloch space. For instance, \eqref{eqn:Hedenmalm} is a consequence of the Azuma--Hoeffding inequality for dyadic martingales; see Lemma \ref{lemma:kristian}. Moreover, the law of iterated logarithm for Bloch martingales led Makarov to an estimate on the mean growth of Bloch functions (see \eqref{eqn:lil}).
        
        In particular, any restriction problem for the Bloch space can be recast as a restriction problem for Bloch martingales, and Question \ref{q:Bloch} is equally interesting - and challenging -  in such discrete probabilistic setting.
    
        \subsection*{Notation} Given two positive functions $f$ and $g$, $f(x) \lesssim g(x)$ denotes that there exists a universal constant $C$ such that $f(x) \leq C g(x)$ for all $x$. If also $g(x) \lesssim f(x)$, we say that $f \simeq g$.

        \subsection*{Acknowledgments}
    
            We would like to express our gratitude to Artur Nicolau for helpful discussions, especially concerning Lemma~\ref{lemma:MisBHO} and the sharpness of our second main theorem, namely, Theorem \ref{theo:necessary} (iii). Furthermore, we are grateful to the helpful anonymous referees who provided feedback on the drafts of this manuscript.
        
            The first author was partially supported by the grant 275113 of the  research Council of Norway through the Alain Bensoussan Fellowship Programme from ERCIM, and by the Emmy Noether Program of the German Research Foundation (DFG Grant 466012782). The work of the second author was funded by grant 275113 of the Research Council of Norway through the Alain Bensoussan Fellowship Programme from ERCIM, and by the postdoctoral scholarship JCK22-0052 from The Kempe Foundations and was partially supported by grant PID2019-106870GB-I00 from Ministerio de Ciencia e Innovación (MICINN). The work of the third author was supported by grant 334466 of the Research Council of Norway, ``Fourier Methods and Multiplicative Analysis''.

    \section{Weights of bounded hyperbolic oscillation} \label{sec:BHO}
    
        The aim of this Section is to gather those properties of weights of bounded hyperbolic oscillation that play a role in the proof of Theorem \ref{theo:main1}. 
    
        First, we note that maximal functions of $L^1$ functions are of bounded hyperbolic oscillation:
    
        \begin{lemma} 
        \label{lemma:MisBHO}
            Let $\Omega$ be a set of positive measure of $\D$. If $f \in L^1(\Omega)$ then $M_\Omega f$ is of bounded hyperbolic oscillation, with constant independent of $f$.
        \end{lemma}
    
        \begin{proof}
            Let $z, \lambda \in T(I)$ for certain arc $I \subset \T$. If $z \in S(J)$ for certain $J \subset \T$, then $\lambda$ is included in the Carleson box generated by $3J$, the interval with same center as $J$ and triple its length. Then
                \[
                    \dfrac{1}{A \big( S(J) \big)} \int_{S(J) \cap \Omega} |f| \, dA \leq \dfrac{9}{A \big (S(3J) \big)} \int_{S(3J) \cap \Omega} |f| \, dA \leq 9 M_\Omega f (\lambda),
                \]
                so $M_\Omega f(z) \leq 9 M_\Omega f (\lambda)$ if $z, \lambda \in T(I)$.
        \end{proof}
    
        It is worth pointing out that a weight which is of bounded hyperbolic oscillation is a $B_p$ weight if and only if it is an $A_p$ weight in the unit disc. Restricted $A_p$ weights on a measurable subset $\Omega\subset\D$ are defined from the Hardy-Littlewood maximal function
            \[
                \mathcal{M}^b_\Omega w(z):= \sup_{z\in B} \frac{1}{A(B)} \int_{B \cap \Omega} w \, dA, 
            \]
            where the supremum is taken over all the Euclidean balls  with center in $\D$ containing $z$. A weight $w$ is in the $A_{1, \Omega}$ class if $\mathcal{M}^b_\Omega w\le C w$ almost everywhere on $\Omega$. For $p>1$, $w$ is an $A_{p, \Omega}$ weight if 
            \[
                [w]_{A_{p, \Omega}}:= \sup_{z\in B}\dfrac{1}{A(B)}\int_{B \cap \Omega} w \, dA \left( \dfrac{1}{A (B)} \int_{B \cap \Omega} w^{-\frac{1}{p-1}} \, dA \right)^{p-1} 
            \]
            is finite. Clearly, $M_\Omega w\lesssim\mathcal{M}_\Omega^bw$ for all weights $w$. On the other hand, $\mathcal{M}^b_\Omega w$ can be controlled by $M_\Omega$ and the constant of bounded hyperbolic oscillation of $w$:
            
            \begin{lemma} \label{lemma:BpAp}
                Consider a set $\Omega$ of positive area measure. Let $w$ be a weight of bounded hyperbolic oscillation and $p \geq 1$. Then, $w \in B_{p, \Omega}$ if and only if $w$ is an $A_{p, \Omega}$ weight of $\D$.
            \end{lemma}
    
            \begin{proof}
                Let $p \geq 1$. We just need to check that $w \in B_{p, \Omega}$ is an $A_{p, \Omega}$-weight as well. Let $B$ be the intersection of an Euclidean ball  with center in $\D$ and the unit disc, and let $I \subset \T$ be of minimal length such that $B \subset S(I)$. If $B \subset T(I)$, then
                    \[
                        \dfrac{1}{A(B)}\int_{B \cap \Omega} w \, dA \left( \dfrac{1}{A (B)} \int_{B \cap \Omega} w^{-\frac{1}{p-1}} \, dA \right)^{p-1} \leq \dfrac{\underset{z\in B \cap \Omega}{\sup} w}{\underset{z \in B\cap \Omega }{\inf}w} \leq L_w,
                    \]
                    where $L_w$ is the constant of bounded hyperbolic oscillation of $w$.

                If $B\not\subseteq T(I)$, then $A(B)$ is comparable to $A \big( S(I) \big)$, by the minimality of $I$  and the fact that the center of $B$ lies in $\D$. Therefore
                    \[
                        \dfrac{1}{A(B) } \int_{B \cap \Omega} w \, dA \left( \dfrac{1}{A (B)}  \int_{B \cap \Omega} w^{-\frac{1}{p-1}} \, dA \right)^{p-1} \le C_p [w]_{B_{p,\Omega}}.
                    \]
                    This shows that $[w]_{A_{p, \Omega}}\leq C_{p, L_w} [w]_{B_{p, \Omega}}$ for all $p>1$.
                    When $p = 1$, the same argument yields that
                    \[
                     M_\Omega w(z) \lesssim \mathcal{M}^b_\Omega w(z) \lesssim \max \{L_w, M_\Omega w(z) \}.
                    \]
                    Then the two maximal functions are comparable on weights of bounded hyperbolic oscillation, and $w \in A_{1, \Omega}$ if and only if $w \in B_{1, \Omega}$.
            \end{proof}
    
    \section{Factorization of restricted weights}
    \label{sec:factorization}
             
        The aim of this Section is to prove a factorization result within the class of $B_p$ weights which are of bounded hyperbolic oscillation on a measurable $\Omega \subset \D$. The first step in doing so is to describe weights in $B_{p, \Omega}$ as those that characterize the weak boundedness of the associated maximal function $M_\Omega$. To this end, we will use dyadic tools from modern harmonic analysis.
                
        Let $\dyad_\theta$, $0 \leq \theta \leq 1$, be the collection of all of the arcs of the form
            \[
                I = \left \{ e^{2\pi it} \: : \: \dfrac{j}{2^k} \leq t - \theta < \dfrac{j+1}{2^k} \right \}, \quad 0 \leq j < 2^k, k \geq 0.
            \]
            These are usually called the \emph{dyadic grids} of $\T$. Given a dyadic grid $\dyad_\theta$, we will write $[w]_{B_{p,\Omega,\theta}}$ and $M_{\Omega, \theta}$ to denote the dyadic version of the Békollé-Bonami characteristic and the maximal function defined in \eqref{RestrictedB_p} and \eqref{RestrictedMaximalFunction}, respectively. Namely, 
            \[
                [w]_{B_{p,\Omega, \theta}} := \sup_{I \in\dyad_\theta} \dfrac{1}{A \big( S(I) \big)}\int_{S(I) \cap \Omega} w \: dA \left( \dfrac{1}{A \big( S(I) \big)} \int_{S(I) \cap \Omega} w^{-\frac{1}{p - 1}} \: dA \right )^{p - 1}
            \]
            and
            \[
                M_{\Omega, \theta} w (z):=  \sup_{S(I) \ni z,\,\, I\in\dyad_\theta} \dfrac{1}{A \big( S(I) \big)} \int_{S(I) \cap \Omega} w \, dA 
            \]           
            We write $[w]_{B_1,\Omega, \theta}$ to denote the infimum of the constants $C$ for which $M_{\Omega, \theta}w\le C w$ almost everywhere in $\Omega$. For all $p\ge1$, a weight $w$ is in the $B_{p, \Omega, \theta}$ class if $[w]_{B_{p, \Omega, \theta}}<\infty$.
                   
        It is clear that $B_{p, \Omega} \subset B_{p, \Omega, \theta}$ for all $0\le\theta<1$. We can use an observation by Hytönen and Pérez \cite{HytonenPerez} (even though it seems it can be originally attributed to Garnett and Jones \cite{GarnettJones} and Christ) to see that $B_{p, \Omega}$ is the finite intersection of certain dyadic Békollé-Bonami weight classes. In order to do so, we require the following covering lemma. We refer the reader to \cite{MR2979613} for further information regarding dyadic coverings.
        
        \begin{lemma}
        \label{lemma:covering}
            For all $I\subset\T$ there exists a $J\in\dyad_0\cup\dyad_{1/3}$ such that $I\subset J$ and $|J|\le8|I|$.
        \end{lemma}
        
        \begin{proof}
            Let $k\in\N$ be chosen so that $2^{-k-1}<|I|\le 2^{-k}$, and let's argue by contradiction. If $I$ is not included in any interval of $\dyad_0 \cup \dyad_{1/3}$ of length $2^{-k+2}$, then there are some $j_1, j_2 \in \{0, \ldots, 2^{k-2} - 1\}$ such that $e^{2\pi i \frac{j_1}{2^{k-2}}}, e^{2\pi i \left( \frac{1}{3} + \frac{j_2}{2^{k-2}}\right)} \in I$. Thus we have that
                \[
                    |I| \geq \min_{0 \leq k < 2^{k-2}} \left | \dfrac{1}{3} - \dfrac{j}{2^{k-2}} \right | = \min_{0 \leq k < 2^{k-2}} \left | \sum_{l=1}^\infty \dfrac{1}{2^{2l}} - \dfrac{j}{2^{k-2}} \right | = \sum_{2l > k-2} \dfrac{1}{2^{2l}} > 2^{-k},
                \]
                reaching the desired contradiction.
        \end{proof}
    
        This implies that, while the dyadic Bekolle-Bonami classes are bigger than the continuous ones, by intersecting two particular ones we recover classic $B_p$ weights:
    
        \begin{coro} \label{coro:1/3}
            If $p \ge 1$, then $B_{p, \Omega} = B_{p, \Omega, 0} \cap B_{p, \Omega, 1/3}$. Moreover, 
                \begin{equation} \label{eqn:Twodyadic}
                    M_\Omega f (z) \simeq M_{\Omega,0} f(z) + M_{\Omega, 1/3} f(z)
                \end{equation}
                for all $f$ in $L^1(\D)$.
        \end{coro}
    
        We are now ready to show that $B_{p, \Omega}$ and its dyadic versions are the weights that characterize the weak boundedness of the their corresponding maximal function. The same can be said about $A_{p, \Omega}$ regarding $\mathcal{M}^b_\Omega$, but we will not explicitly prove it because it will not play a role in our argument. Since these three classes of weights are distinct, these weak boundedness properties can be understood as a fundamental difference between the maximal functions $M_\Omega$, $M_{\Omega, \theta}$ and $\mathcal{M}^b_\Omega$.    
    
        \begin{theo} \label{CharacterBpD}
            Let $w \in L^1 (\Omega) \setminus \{ 0 \}$ be a non-negative weight, and let $p > 1$. If $w \in B_{p, \Omega}$, then the weak-type inequality
                \begin{equation}\label{MaximalWeakType}
                    w \big( \{ z \in \Omega : M_{\Omega} f(z) > \lambda \} \big) \lesssim \dfrac{[w]_{B_{p, \Omega}}}{\lambda^p} \| f \|_{L^p(\Omega, w)}^p,
                \end{equation}
                where $w(E):= \int_E w \, dA$, holds for all $\lambda > 0$ and $f \in L^p (\Omega, w)$. Conversely, if the restricted maximal operator is of weak-type $(p,p)$ with respect to $w$, then $w$ is a $B_{p, \Omega}$-weight. 
        \end{theo}
    
        \begin{proof}
            In light of Corollary~\ref{coro:1/3}, it is enough prove the dyadic version of the statement.
            
            First, assume that $M_{\Omega, \theta}$ is of weak-type $(p,p)$. Take $I \in \dyad_\theta$ such that $w \big( S(I) \cap \Omega \big) > 0$. If $K \subset S(I) \cap \Omega$ is a measurable set with positive area, we see that
                \[
                    w \big ( S(I) \cap \Omega \big) \leq w \left ( \left \{ z \in \Omega : M_{\Omega, \theta} \chi_K (z) \geq \dfrac{A(K)}{A \big( S(I) \big)} \right \} \right ) \leq C \left( \dfrac{A \big( S(I) \big)}{A(K)} \right)^p w(K).
                \]
                Choosing $I = \T$, we realize that $w$ cannot vanish in a set of positive area.

            Testing with the function 
                \[
                    f_n := w^{-1/(p-1)} \chi_{S(I) \cap \Omega \cap \{ w^{-1/(p-1)} \leq n \}}
                \]
                and the number $\lambda_n = \frac{1}{A ( S(I) )} \int_{S (I) \cap \Omega} f_n \: dA$ in \eqref{MaximalWeakType}, we see that
                \[
                    \dfrac{1}{A \big( S(I) \big)}\int_{S(I) \cap \Omega} w \: dA \left( \dfrac{1}{A \big( S(I) \big)} \int_{S (I) \cap \Omega \cap \{ w^{-1/(p-1)} \leq n \}} w^{-\frac{1}{p - 1}} \: dA \right)^{p - 1} \leq C, \quad \forall n \geq 1.
                \]
                Therefore the monotone convergence theorem and the fact that $w > 0$ almost everywhere yield that
                \[
                    \dfrac{1}{A \big( S(I) \big)}\int_{S(I) \cap \Omega} w \: dA \left( \dfrac{1}{A \big( S(I) \big)} \int_{S (I) \cap \Omega} w^{-\frac{1}{p - 1}} \: dA \right)^{p - 1} \leq C, \quad \forall I \in \dyad_\theta,
                \]
                and hence $w$ is a restricted dyadic Békollé--Bonami weight.
              
            Now, assume $w \in B_{p, \Omega, \theta}$, take $f \in L^p (\Omega, w)$, let $\lambda > 0$, and consider the superlevel set $E_\lambda := \{ z \in \Omega \: : \: M_{\Omega, \theta} f(z) > \lambda \}$. If $z \in E_\lambda$, then there exists $I \in \dyad_\theta$ with $z \in S(I) \cap \Omega$ and
                \[
                    \dfrac{1}{A \big( S(I) \big)} \int_{S(I) \cap \Omega} |f| \: dA > \lambda,
                \]
                so that $S(I) \cap \Omega \subset E_\lambda$. Therefore, by the dyadic structure, $E_\lambda = \cup_{I \in \mathcal{J}} \big( S(I) \cap \Omega \big)$ for a disjoint collection $\mathcal{J} \subset \dyad_\theta$ satisfying that
                \[
                    \dfrac{1}{A \big ( S(I) \big) } \int_{S(I) \cap \Omega} |f| \: dA > \lambda, \qquad I \in \mathcal{J}.
                \]
                Therefore Hölder's inequality implies that
                \begin{eqnarray*}
                    w \big( S(I) \cap \Omega \big) & \leq & \dfrac{\int_{S(I) \cap \Omega} |f|^p w \: dA}{\lambda^p} \dfrac{1}{A \big( S(I) \big)} \int_{S(I) \cap \Omega} w \: dA \left( \dfrac{1}{A \big( S(I)  \big)} \int_{S(I) \cap \Omega} w^{-\frac{1}{p-1}} \: dA \right)^{p-1} \\
                        & \leq & \dfrac{[w]_{B_{p, \Omega,\theta}}}{\lambda^p} \int_{S(I) \cap \Omega} |f|^p w \: dA
                \end{eqnarray*}
                for all $I \in \mathcal{J}$, and hence \eqref{MaximalWeakType} holds for the maximal function $M_{\Omega, \theta}$. We conclude the proof by applying \eqref{eqn:Twodyadic}.
        \end{proof}
    
        \begin{rem} \label{rem:maxopbdd}
            When $\Omega = \D$, $w \in B_{p}$ if and only if $M$ is \emph{bounded} on $L^p(\D, w)$. The same can be said regarding the dyadic case. See, for instance, \cite{Lerner}.
        \end{rem}
                
        Next we will prove a factorization theorem for $B_{p, \Omega}$.  This result is a consequence of Rubio de Francia's factorization techniques \cite{RubiodeFranciaArticle}. Nevertheless, we include a constructive proof whose original idea is due to Coifman, Jones and Rubio de Francia \cite{CoifmanJonesRubiodeFrancia}, as we will need to inspect the details of this proof when considering factorization within the class of weights which are of bounded hyperbolic oscillation.
    
        \begin{theo} \label{FactorizationRubiodeFrancia}
            Let $\Omega \subset \D$ a measurable set with positive area and $p>1$. If $w \in L^1 (\Omega) \setminus \{ 0 \}$ satisfies that $M_{\Omega}$ is bounded on $L^p(\Omega, w)$ and $L^{\frac{p}{p-1}} (\Omega, w^{-\frac{1}{p-1}})$, then there exist $B_{1, \Omega}$ weights $w_1$ and $w_2$ such that $w = w_1 w_2^{1 - p}$ and
                \[
                    [w_1]_{B_{1, \Omega}}, [w_2]^{p-1}_{B_{1, \Omega}} \leq 2 \left( \| M_{\Omega} \|_{L^\frac{p}{p-1}\left(\Omega, w^{-\frac{1}{p-1}}\right)} + \| M_{\Omega} \|_{L^p (\Omega, w)}^{p-1} \right).
                \]
                Moreover, if $w$ is of bounded hyperbolic oscillation, then so are $w_1$ and $w_2$. 
        \end{theo}
        
        \begin{proof}
   
            Assume that $p \in (1,2]$ (if not, we can factor $w^{-\frac{1}{p-1}} \in B_{\frac{p}{p-1}, \Omega}$ instead). Consider the operator defined as
                \begin{equation}\label{AuxiliarOperatorS}
                    S(f) := M_{\Omega} (fw) w^{-1} + M_{\Omega}^{p-1} \left( |f|^\frac{1}{p-1}\right).
                \end{equation}
    
            It is plain that $M_{\Omega}$ is a positive and sublinear operator. Since $1 < p \leq 2$, it is not difficult to check that $S$ is also positive and sublinear. Moreover, $S$ is bounded from $L^{\frac{p}{p-1}} (\Omega, w)$ to itself. Indeed, if $f \in L^{\frac{p}{p-1}} (\Omega, w)$ then $|f|^{\frac{1}{p-1}} \in L^p (\Omega, w)$ and $fw \in L^{\frac{p}{p-1}} (\Omega, w^{-\frac{1}{p-1}})$, so we have that
                \[
                    \left \| M_{\Omega}^{p-1} \left( |f|^{\frac{1}{p-1}}\right) \right \|_{L^{\frac{p}{p-1}} (\Omega, w)} \leq \| M_{\Omega} \|_{L^p (\Omega, w)}^{p-1} \| f \|_{L^{\frac{p}{p-1}} (\Omega, w)}
                \]
                and
                \[
                    \left \| M_{\Omega} (fw) w^{-1} \right \|_{L^{\frac{p}{p-1}} (\Omega, w)} \leq \| M_{\Omega} \|_{L^\frac{p}{p-1} \left(\Omega, w^{-\frac{1}{p-1}} \right)} \| f \|_{L^{\frac{p}{p-1}} (\Omega, w)}.
                \]
    
            Observe also that, by definition, the operator $S$ satisfies
                \[
                    M_{\Omega}^{p-1} \left( |f|^\frac{1}{p-1} \right) \leq S(f) \quad \mbox{ and } \quad M_{\Omega} (fw) \leq S(f)w,
                \]
                for all $f \in L^\frac{p}{p-1} (\Omega, w)$.
    
            Now, consider the function
                \begin{equation} \label{AuxiliarFunctionFactorization}
                    f := \sum_{k = 0}^\infty \dfrac{S^k(u)}{(2 \| S \|)^k} \in L^\frac{p}{p-1} (\Omega, w),
                \end{equation}
                where $u$ is any positive $L^\frac{p}{p-1}(\Omega, w)$ function (for instance, we might fix $u \equiv 1$). Since $S$ is positive and sublinear, an immediate computation shows that
                    \[
                        S(f) \leq \sum_{k = 1}^\infty \dfrac{S^{k+1}(u)}{(2 \| S \|)^k} = 2 \| S \| \sum_{m = 2}^\infty \dfrac{S^{m}(u)}{(2 \| S \|)^m} \leq 2 \| S \| f
                    \]
                    almost everywhere in $\Omega$. Therefore he have that
                    \[
                        M_{\Omega} \left( f^\frac{1}{p-1} \right) \leq (2 \| S \|)^\frac{1}{p-1} f^\frac{1}{p-1} \quad \mbox{ and } \quad M_{\Omega} (fw) \leq 2 \| S \| fw.
                    \]
                    In other words, we have that $w_1 := f w$ and $w_2 := f^{\frac{1}{p-1}}$ belong to $B_{1, \Omega}$ and moreover
                    \begin{eqnarray*}
                        [w_1]_{B_{1, \Omega}} & \leq & 2 \left( \| M_{\Omega} \|_{L^\frac{p}{p-1}\left(\Omega, w^{-\frac{1}{p-1}}\right)} + \| M_{\Omega} \|_{L^p (\Omega, w)}^{p-1} \right), \\
                        {[w_2]}_{B_{1, \Omega}} & \leq & 2^\frac{1}{p-1} \left( \| M_{\Omega} \|_{L^\frac{p}{p-1}\left(\Omega, w^{-\frac{1}{p-1}}\right)} + \| M_{\Omega} \|_{L^p (\Omega, w)}^{p-1} \right)^\frac{1}{p-1}.
                    \end{eqnarray*}
                    Since $w = w_1 w_2^{1-p}$, we have obtained the desired factorization.

            Note that, at this point, we have not required the fact that $w$ is of bounded hyperbolic oscillation. Without this assumption, of course the factors may fail to satisfy this property.
        
            Now suppose that $w$ is of bounded hyperbolic oscillation. In light of Lemma~\ref{lemma:MisBHO}, observe that for any positive $u$ the operator $S(u)$ introduced in \eqref{AuxiliarOperatorS} is of bounded hyperbolic oscillation as well, with constant less than or equal to $9L_w$. Hence, the function defined in \eqref{AuxiliarFunctionFactorization} satisfies $L_f \leq 9L_w$, and therefore Rubio de Francia's factorization result also holds within this subclass of weights.
        \end{proof}
    
    \section{Proof of Theorem \ref{theo:main1}} 
    \label{sec:proof}
    
        The first ingredient of the proof of Theorem \ref{theo:main1} is the following lemma, whose proof is mostly based on Theorem 3.4 in Chapter II of \cite{Garcia-CuervaRubiodeFrancia}, although the original idea goes back to \cite{CoifmanRochberg}.
                
        \begin{lemma} \label{B1ConditionMaxFunct}
           Let $w$ be an integrable weight in $\D$. Then, for every $\gamma \in (0,1)$ we have that $(M w)^\gamma \in B_{1}$. Furthermore, $[(M w)^\gamma]_{B_{1}} \lesssim C$ for some $C = C(\gamma)$.
        \end{lemma}
    
        \begin{proof}
            First, we will show that
                \[
                    M_{0} (M w)^\gamma \lesssim (M w)^\gamma
                \]
                in $\D$, with a constant depending only on $\gamma$. Let $I \in \dyad_0$ and $z \in S(I)$. Consider $J_1, J_2 \in \dyad_{1/3}$ of the same generation such that $S(I) \subset S(J_1) \cup S(J_2)$. Define the weights $w_1 = w \chi_{S(J_1) \cup S(J_2)}$ and $w_2 = w - w_1$.
            
            On the one hand, we have that
                \[
               \dashint_{S(I)} (M w_1)^\gamma \: dA = \dfrac{\gamma}{A \big( S(I) \big)} \int_0^\infty t^{\gamma - 1} A \big( \{ z \in S(I) : M w_1 (z) > t \} \big) \: dt.
                \]
            
            Observe that
            \[
                \dfrac{\gamma}{A \big( S(I) \big)} \int_0^{\frac{w ( S(I) )}{A( S(I) )}} t^{\gamma - 1} A \big( \{ z \in S(I) : M w_1 (z) > t \} \big) \: dt \leq \left( \dfrac{w \big( S(I) \big)}{A \big( S(I) \big)}\right)^\gamma
            \]
            and since $M$ is of weak-type $(1,1)$ (with respect to the area measure) we deduce that
            \begin{eqnarray*}
                \dfrac{\gamma}{A \big( S(I) \big)} \int_{ \frac{w ( S(I) )}{ A ( S(I) )}}^\infty t^{\gamma - 1} A \big( \{ z \in S(I) : M w_1 (z) > t \} \big) \: dt & \leq & \dfrac{\gamma \| w_1 \|_{L^1(\D)}}{A \big( S(I) \big)} \int_{ \frac{w ( S(I) )}{ A ( S(I) )}}^\infty t^{\gamma - 2} \: dt \\
                & \leq & \dfrac{\gamma}{1 - \gamma} \dfrac{\| w_1 \|_{L^1 (\D)}} {A \big( S(I) \big)} \left( \dfrac{w \big( S(I) \big)}{A \big( S(I) \big)}\right)^{\gamma-1}.
            \end{eqnarray*}

            Let $\zeta_1, \zeta_2 \in T(I)$ such that $\zeta_1 \in S(J_1)$ and $\zeta_2 \in S(J_2)$. Without loss of generality, we may choose that $\zeta_1$ has the same argument as $z$, so that $M w (\zeta_1) \leq M w (z)$ (see Figure~\ref{fig:ChoiceOfZeta}).
            
            \begin{figure}[h!]
                \centering
                \begin{tikzpicture}[x=3cm,y=3cm]
                    \def \thet {60};
                    \def \maxlevel {4};
                    \def \mod {0.9};
                    \def \arg {-10};
                    \pgfmathsetmacro{\angle}{180*(1-\mod)}
            
                    \draw[thick] (0,0) circle (1);
                    \foreach \i in {1,...,\maxlevel}
                    {
                        \pgfmathsetmacro{\radtemp}{1-2^(-\i)}
                        \pgfmathsetmacro{\rangecounter}{2^\i}
                        \draw[opacity=0.3] (0,0) circle (\radtemp);
                        \foreach \j in {1,...,\rangecounter}
                        {
                            \draw[opacity=0.3] ($ \radtemp*({cos(\thet + \j*360/\rangecounter)},{sin(\thet + \j*360/\rangecounter)}) $) -- ($  ({cos(\thet + \j*360/\rangecounter)},{sin(\thet + \j*360/\rangecounter)}) $);
                        }
                    }
                    
                    \draw[red] ($ \mod*({cos(\arg)},{sin(\arg)}) $) node{$\bullet$}node[above, left]{$z$};
                    \draw[dashed, red, thick] (0:0.75)--(0:1);
                    \draw[red, thick] (0:0.75) arc(360:270:0.75)--(270:1);
                    \fill[red, opacity=0.3] (270:0.75) arc(270:360:0.75)--(360:1) arc(360:270:1)--cycle;
                    \draw[blue] ($ 0.75*({cos(\arg)},{sin(\arg)}) $) node{$\bullet$}node[above left]{$\zeta_1$};
                    \draw[dashed, blue, thick] (60:0.75)--(60:1);
                    \draw[blue, thick] (60:0.75) arc(420:330:0.75)--(330:1);
                    \fill[blue, opacity=0.15] (330:0.75) arc(330:420:0.75)--(420:1) arc(420:330:1)--cycle;
            
                    \draw[blue] ($ 0.8*({cos(\arg-30)},{sin(\arg-30)}) $) node{$\bullet$}node[above left]{$\zeta_2$};
                    \draw[blue, thick] (-30:0.75) arc(330:240:0.75)--(240:1);
                    \fill[blue, opacity=0.15] (240:0.75) arc(240:330:0.75)--(330:1) arc(330:240:1)--cycle;
            
                    \draw ($ ({cos(135)},{sin(135)}) $) node[above=6pt]{\LARGE $\mathbb{T}$};
                \end{tikzpicture}
                \caption{Example of {\color{red} $S(I)$} and {\color{blue} $S(J_1) \cup S(J_2)$}.}
                \label{fig:ChoiceOfZeta}
            \end{figure}
            
            Since $\| w_1 \|_{L^1(\D)} = w \big( S(J_1) \big) + w \big( S(J_2) \big)$ we have that
            \[
                \dfrac{\| w_1 \|_{L^1(\D)}}{A \big( S(I) \big)} \leq M w (\zeta_1) + M w (\zeta_2)
            \]
            and hence Lemma \ref{lemma:MisBHO} yields that
            \[
                \dashint_{S(I)} \big( M w_1 )^\gamma \, dA \lesssim \dfrac{\gamma}{1 - \gamma} \big( M w (z) \big)^\gamma.
            \]

            Since $w_2$ is supported outside $S(J_1) \cup S(J_2)$, we have that
            \[
                \dashint_{S(I)} \big( M_{0} w_2 \big)^\gamma \, dA = \big( M_{0} w_2 \big)^\gamma (z) \leq \big( M w (z) \big)^\gamma
            \]
            and
            \[
                \dashint_{S(I)} \big( M_{1/3} w_2 \big)^\gamma \, dA = \dfrac{A \big ( S(J_1) \cap S(I) \big) \big( M_{1/3} w_2 (\zeta_1) \big)^\gamma + A \big ( S(J_2) \cap S(I) \big) \big( M_{1/3} w_2 (\zeta_2) \big)^\gamma}{A \big ( S(I) \big)}.
            \]
            Again Lemma \ref{lemma:MisBHO} yields that
            \[
                \dashint_{S(I)} \big( M_{1/3} w_2 \big)^\gamma \, dA \lesssim \big( M w (z) \big)^\gamma
            \]
            and therefore \eqref{eqn:Twodyadic} implies that there exists $C = C(\gamma)$ such that
            \[
                \dashint_{S(I)} \big( Mw \big)^\gamma \, dA \leq C \big( Mw (z) \big)^\gamma.
            \]

            Summing up, we have deduced that $M_{0} (Mw)^\gamma \leq C(\gamma) (Mw)^\gamma$ in $\D$. Replacing $\dyad_0$ by $\dyad_{1/3}$, \eqref{eqn:Twodyadic} implies again that $(Mw)^\gamma \in B_1$ with a constant depending only on $\gamma$.
        \end{proof}
                
        The following result highlights a deviation between $B_1$ and the classical setting of $A_1$.
    
        \begin{coro}
            There exists $w \in B_1 \setminus \{ (M f)^\gamma : f \in L^1(\D),  0 < \gamma < 1 \}$. Necessarily, this weight is not of bounded hyperbolic oscillation.
        \end{coro}
    
        We are ready to prove our main extension theorem, Theorem~\ref{theo:main1}.  Observe that thanks to Theorem \ref{FactorizationRubiodeFrancia}, it is enough to prove the case $p=1$. Indeed, let $p>1$ and suppose that $w^q$ is a $B_{p, \Omega}$ weight for some $q>1$. Then, given $1<t<q$, $w^t$ satisfies the strong $L^p$ estimates required in Theorem \ref{FactorizationRubiodeFrancia}, thanks to Theorem \ref{CharacterBpD} and Jensen's inequality. Therefore, $w^t=w_1w_2^{1-p}$, for some $w_1$ and $w_2$ in $B_{1, \Omega}$ of bounded hyperbolic oscillation. Assuming Theorem \ref{theo:main1} for $p=1$, we can extend both $w_1^{1/t}$ and $w_{2}^{1/t}$ to $B_1$ weights $W_1$ and $W_2$ of bounded hyperbolic oscillation. Thus $W_1W_2^{1-p}$ is a $B_p$ weight of bounded hyperbolic oscillation that extends $w$.
    
        The implication (\ref{Wolff2)})$\Longrightarrow$(\ref{Wolff1)}) is just the self-improvement property for $A_p$ weights (see \cite{AlemanPottReguera} as well). We are therefore left to show that (\ref{Wolff1)})$\implies$(\ref{Wolff2)}) for $p=1$: 
    
        \begin{theo} \label{WolffB1}
            Let $\Omega$ be a set of positive area of $\D$, and let $w$ be a weight in $\Omega$ of bounded hyperbolic oscillation. If $w^q \in B_{1, \Omega}$ for some $q > 1$, then there exists $W \in B_{1}$, of bounded hyperbolic oscillation, such that $W(z) = w(z)$ for almost every $z \in \Omega$.
        \end{theo}
    
        \begin{proof}
            Since $w$ is of bounded hyperbolic oscillation,  from Lemma \ref{lemma:BpAp} we have that
                \[
                    w^q(z) \leq \mathcal{M}^b_\Omega (w^q) (z) \leq K_w M_\Omega (w^q) (z)
                \]
                for almost every $z \in \Omega$. Considering $w$ as a weight defined on $\D$ (setting $w(z) := 0$ for all $z \in \D \setminus \Omega$), we see that
                \[
                    M (w^q) (z) = M_{\Omega} (w^q) (z) \leq [w^q]_{B_{1, \Omega}} w^q (z)
                \]
                for almost every $z \in \Omega$. In other words, we have seen that there exists a function $k$ in $\Omega$, with $\log k$ bounded by a quantity depending only on $w$ and $q$, such that
                \[
                    w (z) = \big( k(z) M (w^q) (z) \big)^{\frac{1}{q}} =: W(z)
                \]
                for almost every $z \in \Omega$. Finally, setting $k (z) := 1$ in $z \in \D \setminus \Omega$ we have that $W$ is of bounded hyperbolic oscillation (because of Lemma~\ref{lemma:MisBHO}) and it is $B_1$ (by Lemma \ref{B1ConditionMaxFunct}). 
        \end{proof}
    
    \section{Restrictions of Bloch functions and martingales}
    \label{sec:martingales}
    
        We first recall the probabilistic framework, due to Makarov \cite{makarov}, connecting Bloch functions and dyadic martingales. To this end, in this section, let $\dyad_k$ be the set of dyadic sub-intervals of $[0, 1]$ of length $2^{-k}$,
            \[
                \dyad_k:=\{[j/2^{-k}, (j+1)/2^{-k})\,:\, j=0, \dots, 2^{k}-1 \},
            \]
            and let $\dyad=\bigcup_{k=0}^\infty\dyad_k$ be the collection of all dyadic intervals. A dyadic martingale is a sequence of random variables $M=(M_n)_{n\in\N}$, defined on $[0, 1]$, such that
            \begin{description}
                \item[i)] $M_k$ is constant on each interval $I \in \dyad_k$ - we denote the value by $M_I$ 
                \item[ii)]whenever $I=I_1\cup I_2$, $I \in \dyad_k$, $I_j\in\dyad_{k+1}$ for $j=1, 2$, it holds that
                    \[
                        M_{I}=\frac{1}{2}\left(M_{I_1}+M_{I_2}\right).
                    \]
            \end{description}
            Depending on the context, we will refer to $M$ either as the sequence $(M_n)_{n\in\N}$ or the family $(M_I)_{I\in\dyad}$. When identifying $\dyad$ with a dyadic tree, one can view a dyadic martingale as a function on the tree nodes, whose value at each node is the average of its values at that node's children. 
        
        Makarov first observed that given a Bloch function $b$ and a dyadic interval $I$, the limit
            \begin{equation}
            \label{eqn:bI}
                b_I:=\lim_{r\to1}\frac{1}{|I|}\int_Ib(re^{ 2\pi i\theta})\, d\theta
            \end{equation}
            exists, and the collection $(b_I)_{I\in\dyad}$ defines a dyadic martingale on $[0, 1]$ thanks to holomorphicity. Moreover, there exists a universal constant $C$ such that
            \begin{equation}
                |b_I-b(z)|<C\|b\|_\B, \qquad z\in T(I). 
            \end{equation}
            Therefore, since $b$ is Lipschitz with respect to the hyperbolic metric of $\D$, the martingale $(b_I)_{I\in\dyad}$ satisfies the following:
            \begin{equation}
            \label{eqn:blochmartingale}
                |b_I-b_J|<C\|b\|_\B
            \end{equation}
            for every pair of adjacent dyadic intervals $I$ and $J$ of the same length. In particular, $(b_I)_{I\in\dyad}$ has uniformly bounded increments:
            \[
                |b_I-b_{I'}|<C'\|b\|_{\B}
            \]
            for all $I\in\dyad_n$ and $I'\subset I$ in $\dyad_{n+1}$. Moreover, it turns out that \eqref{eqn:blochmartingale} characterizes those dyadic martingales which arise from Bloch functions:
            
        \begin{lemma}[\hspace{1sp}{\cite[Lemma~2.1]{makarov}}]
            A real dyadic martingale $(S_I)_{I\in\dyad}$ satisfies \eqref{eqn:blochmartingale} if and only if there exists a Bloch function $b$ such that $\mre b_I=S_I$ for all $I\in\dyad$.
        \end{lemma}
        
        A real dyadic martingale satisfying \eqref{eqn:blochmartingale} is hence known as a \emph{Bloch martingale}. 
        
        The next Example points out how dyadic martingales with uniformly bounded increments need not to satisfy condition \eqref{eqn:blochmartingale}:
    
        \begin{example}[The random walk and Kahane's martingale]
        \label{ex:kahane}
            For any $x$ in $[0, 1]$, let $(\varepsilon_n(x))_{n\in\N}$ be its dyadic expansion, so that 
                \[
                    x=\sum_{n=1}^\infty \varepsilon_n(x)2^{-n}, \qquad \varepsilon_n(x)\in\{0, 1\}.
                \]
                For any $n \in \N$, let $Y_n\colon[0, 1]\to\{-1, 1\}$ be the random variable defined as
                \[
                    Y_n(x):=2\varepsilon_n(x)-1.
                \]
                Then the sequence $(S_n)_{n\in\N}$ defined as
                \[
                    S_n:=\sum_{k=1}^nY_k
                \]
                is a dyadic martingale, which corresponds to the \emph{random walk} on the dyadic tree. It is a dyadic martingale with bounded increments, but it is not a Bloch martingale. Indeed, if $I_k=[1/2-2^{-k}, 1/2)$ and $J_k=[1/2, 1/2+2^{-k})$, then for all $k\geq 2$ we have $S_{I_k}=k-2=-S_{J_k}$, so that \eqref{eqn:blochmartingale} fails.
        
            \begin{figure}[h!]
                \centering
                \resizebox{12cm}{6cm}{
                \begin{tikzpicture}[x=1cm,y=1cm]
                    \def \distance {1};
                    \def \height {1};
                    \def \labels {
                    {-4,-2,-2,0,-2,0,0,2,-2,0,0,2,0,2,2,4},
                    {-3,-1,-1,1,-1,1,1,3},
                    {-2,0,0,2},
                    {-1, 1},
                    {0}};
                    \gdef \currentlevel {6};
        
                    \foreach \i in \labels 
                        {
                            \xdef \currentlevel {\inteval{\currentlevel - 1}};
                            \pgfmathsetmacro{\currentdistance}{\distance*2^(5-\currentlevel)}
        
                            \ifthenelse{\currentlevel>1}{
                                \gdef \counter {0};
                                \foreach \j in \i
                                {
                                    \xdef \counter {\inteval{\counter + 1}};
                                    \ifthenelse{\isodd{\counter}}{
                                        \draw ($ ({\currentdistance*(\counter-1/2)},{-\currentlevel*\height}) $) node{$\bullet$}node[left=1pt]{\tiny \j};
                                        \draw[opacity=0.3]($ ({\currentdistance*(\counter-1/2)},{-\currentlevel*\height}) $) -- ($ ({\currentdistance*\counter},{-(\currentlevel - 1)*\height}) $) ;
                                    }{
                                        \draw ($ ({\currentdistance*(\counter-1/2)},{-\currentlevel*\height}) $) node{$\bullet$}node[right=1pt]{\tiny \j};
                                        \draw[opacity=0.3]($ ({\currentdistance*(\counter-1/2)},{-\currentlevel*\height}) $) -- ($ ({\currentdistance*(\counter-1)},{-(\currentlevel - 1)*\height}) $) ;
                                    } 
                                }
                            }{
                                \draw (8*\distance,-1*\height) node{$\bullet$}node[above=1pt]{\tiny \i};
                            }
                        }
                \end{tikzpicture}
                }
                \caption{The first generations of the random walk.}
                \label{fig:randomwalk}
            \end{figure}
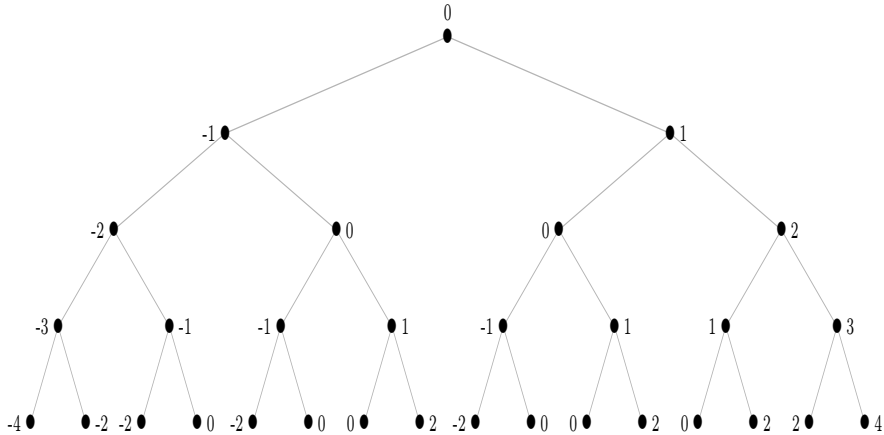
        
            On the other hand, one can obtain a Bloch martingale by switching to $4$-adic expansions. Let $x=\sum_{n=1}^\infty\nu_n(x)4^{-n}$ be the $4$-adic expansion of a point $x$ in $[0, 1]$, $\nu_n(x)\in\{0, 1, 2, 3\}$, and consider the sequence of random variables
                \[
                    W_n(x):=\begin{cases}
                    1\quad&\text{if}\quad\nu_n(x)\in\{0, 3\}\\
                    -1\quad&\text{if}\quad\nu_n(x)\in\{1, 2\}
                    \end{cases}.
                \]
                Then \emph{Kahane's martingale} $K$, $K_n:=\sum_{k=1}^nW_k$ is a Bloch martingale \cite[Prop. 2.2]{makarov}, which can be described as follows: $K_{[0, 1]}=0$, and if $I$ is in $\dyad_k$, then
                \begin{itemize}
                    \item if $k$ is odd, then $K_I=K_{I'}$ where $I'$ is the unique interval in $\dyad_{k-1}$ containing $I$;
                    \item if $k$ is even, then at the left-most and the right-most sub-intervals of $I$ in $\dyad_{k+2}$ the martingale $K$ assumes the value $K_I+1$, whereas at the other two sub-intervals of $I$ in $\dyad_{k+2}$, the martingale $K$ assumes the value $K_I-1$.  
                \end{itemize}
                In particular, the jumps of $K$ between adjacent intervals, as in \eqref{eqn:blochmartingale}, never exceed $2$.
        
            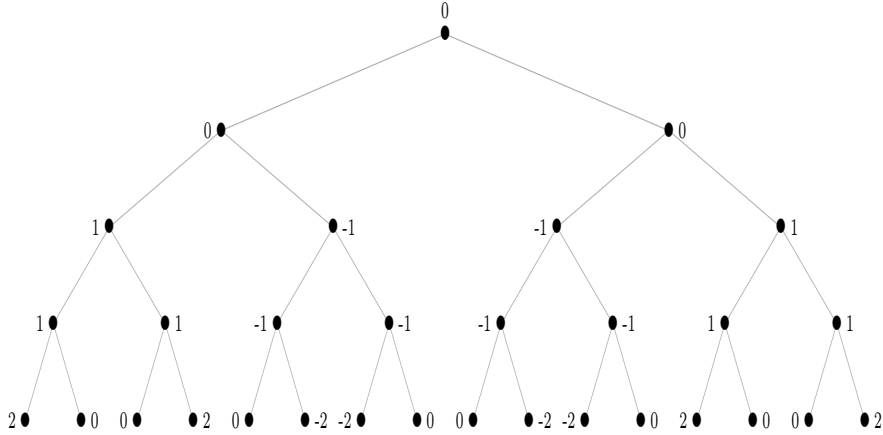
\begin{figure}[h!]
                \centering
                \resizebox{12cm}{6cm}{
                \begin{tikzpicture}[x=1cm,y=1cm]
                    \def \distance {1};
                    \def \height {1};
                    \def \labels {
                    {2,0,0,2,0,-2,-2,0,0,-2,-2,0,2,0,0,2},
                    {1,1,-1,-1,-1,-1,1,1},
                    {1,-1,-1,1},
                    {0,0},
                    {0}};
                    \gdef \currentlevel {6};
        
                    \foreach \i in \labels 
                        {
                            \xdef \currentlevel {\inteval{\currentlevel - 1}};
                            \pgfmathsetmacro{\currentdistance}{\distance*2^(5-\currentlevel)}
        
                            \ifthenelse{\currentlevel>1}{
                                \gdef \counter {0};
                                \foreach \j in \i
                                {
                                    \xdef \counter {\inteval{\counter + 1}};
                                    \ifthenelse{\isodd{\counter}}{
                                        \draw ($ ({\currentdistance*(\counter-1/2)},{-\currentlevel*\height}) $) node{$\bullet$}node[left=1pt]{\tiny \j};
                                        \draw[opacity=0.3]($ ({\currentdistance*(\counter-1/2)},{-\currentlevel*\height}) $) -- ($ ({\currentdistance*\counter},{-(\currentlevel - 1)*\height}) $) ;
                                    }{
                                        \draw ($ ({\currentdistance*(\counter-1/2)},{-\currentlevel*\height}) $) node{$\bullet$}node[right=1pt]{\tiny \j};
                                        \draw[opacity=0.3]($ ({\currentdistance*(\counter-1/2)},{-\currentlevel*\height}) $) -- ($ ({\currentdistance*(\counter-1)},{-(\currentlevel - 1)*\height}) $) ;
                                    } 
                                }
                            }{
                                \draw (8*\distance,-1*\height) node{$\bullet$}node[above=1pt]{\tiny \i};
                            }
                        }
                \end{tikzpicture}
                }
                \caption{The first generations of the Kahane martingale.}
                \label{fig:kahane}
            \end{figure}
            
        \end{example}
        
        Probabilistic tools turn out to be extremely valuable in the study of the mean growth and boundary behaviour of Bloch functions $b$. Makarov showed in \cite[Corollary 3.2]{makarov} that there exists a universal constant $C$ such that 
            \begin{equation}
            \label{eqn:lil}
                \limsup_{r\to1^-}\frac{\mre b(re^{2\pi i\theta})}{\sqrt{\log\frac{1}{1-r}\log\log\log\frac{1}{1-r}}}\leq C\|b\|_\B, \qquad \textrm{a.e. } \theta \in [0,1].
            \end{equation}
            Moreover, inequality \eqref{eqn:Hedenmalm} can for sufficiently small $a$ be derived from the Azuma--Hoeffding concentration inequality for dyadic martingales with bounded increments. We single out the following formulation of the Azuma inequality, as stated in \cite[Lemma 4.7]{kristianinterpolation}.
            
        \begin{lemma}
        \label{lemma:kristian}
            Let $b$ be a Bloch function of norm one, $\|b\|_\B=1$, and for any $I\in\dyad$, let $l_I(\varepsilon, k)$ be the number of dyadic intervals $J$ contained in $I$, of length $2^{-k}|I|$, and such that $|b_J-b_I|>\varepsilon k$. Then, there exist universal constants $\gamma >0$ and $C > 0$ such that 
            \[
                l_I(\varepsilon, k)\leq C2^k~e^{-\gamma\varepsilon^2k}.
            \]
        \end{lemma}
        
        We now have all the necessary tools to prove Theorem \ref{theo:necessary}.
        
        \begin{proof}[Proof of Theorem \ref{theo:necessary}] \,
            \begin{description}
            	\item[(i)] By conformal invariance, it suffices to prove \eqref{eqn:traceblochhede} for $z=0$, uniformly for all Bloch functions of norm less than or equal to $1$. To do so, observe that for $r \in (0, 1)$ and $\lambda>0$, the function
                	\[
                	   g(z):= e^\frac{\lambda|b(z)-b(0)|^2}{-\log(1-r^2)}, \qquad |z|\leq r,
                	\]
            	    is subharmonic. Therefore, if $G_r$ is the Poisson integral of the restriction of $g$ to $\{|z|=r\}$, we have that $G(z)\geq g(z)$ for all $|z|\leq r$, and since $Z = (z_n)_{n\in\N}$ satisfies the Carleson embedding condition \eqref{eqn:cmPoisson},
                	\begin{align*}
                		\sum_{|z_n|<r}e^{\frac{\lambda|b(z_n)-b(0)|^2}{-\log(1-r^2)}}(1-|z_n|^2)
                		&= \int_{|z|<r}g(z)\,d\mu_Z(z) \\
                		&\leq\int_{|z|<r}G(z)\,d\mu_Z(z)
                		\lesssim \int_0^{2\pi}|g(re^{i\theta})|^2\,d\theta.
                	\end{align*}
            	
            	The conclusion then follows from \eqref{eqn:Hedenmalm} upon choosing $\lambda$ sufficiently small.
                
                \item[(ii)] Again, by conformal invariance, it suffices to consider \eqref{eqn:tracebloch} for $z=0$, as long as the estimates we prove are uniform in Bloch functions $b$ such that $\|b\|_{\B}\leq1$. Without loss of generality we may also assume that $b(0)=0$. We are going to use the dyadic decomposition of the unit disc. Since $Z$ satisfies the Carleson condition, $Z$ is a finite union of weakly separated sequences\cite[Proposition 9.11]{AglerMcCarthyBook}. By possibly decomposing each weakly separated subsequence into further finite unions, we may assume that each dyadic top-half contains at most one point of the sequence. Let $G_k = \{T(I) \, : \, I \in \dyad_k\}$ denote the top halves of generation $k$.
            
                Since $Z$ satisfies the Carleson condition, $M = \sum_n (1 - |z_n|^2) < \infty$, and thus 
                    \[ 
                        \#\{n\,:\,1-|z_n|^2 \geq t\} \leq \frac{M}{t}.
                    \]
                    Therefore, it is sufficient to estimate
                    \begin{multline*}
                        \#\left\{n\,:\, t>(1-|z_n|^2)\,,\,e^\frac{\lambda|b(z_n)|^2}{-\log(1-|z_n|^2)}(1-|z_n|^2) > t \right\} \\
                        = \sum_{k\geq\log_2\frac{1}{t}}\#\left\{z_n\in G_k\,:\, |b(z_n)|>k\log2\sqrt{\frac{1}{\lambda}\left(1+\frac{\log~t}{k\log2}\right)}\right\} \\
                        \lesssim \sum_{k\geq\log_2\frac{1}{t}}2^k~e^{-\frac{\gamma k\log^22}{\lambda}\left(1+\frac{\log~t}{k\log2}\right)} = t^{-\frac{\gamma}{\lambda}~\log2}\sum_{k\geq\log_2\frac{1}{t}}\, 2^{k\left(1-\frac{\gamma}{\lambda}~\log2\right)} \simeq \frac{1}{t},
                    \end{multline*}
                    where we used Lemma~\ref{lemma:kristian} and assumed that $\lambda$ was sufficiently small,  $\lambda < \gamma \log 2$.
            
                \item[(iii)] We show that the sum in \eqref{eqn:counter} diverges at $z=0$ for the Bloch function $K$ corresponding to the Kahane martingale and a suitably chosen interpolating sequence $(z_n)_{n\in\N}$. It is important for the counterexample that $K$ satisfies the converse inequality to \eqref{eqn:lil}, see \cite[Corollary~3.5]{makarov}. Namely, there exists a constant $c > 0$ such that
                    \begin{equation}
                    \label{eqn:revlil}
                        \limsup_{r\to1^-}\frac{\mre K(re^{2\pi i\theta})}{\sqrt{\log\frac{1}{1-r}\log\log\log\frac{1}{1-r}}}\geq c, \qquad \textrm{a.e. } \theta \in [0,1].
                    \end{equation}
            
                Let $(s_j)_{j\in\N}$ be an increasing sequence of positive numbers such that $s_j/j \to \infty$ as $j\to\infty$. We construct the desired interpolating sequence recursively. We start the process by selecting points $z_1^1, z_2^1, \dots, z^1_{l_0}$ for which the corresponding Carleson squares $(S(I_{z_i^1}))_{i=1}^{l_0}$ are pairwise disjoint. Because of \eqref{eqn:revlil},  we may choose the points such that
                    \[
                        \frac{|K(z^1_i)|^2}{\log\frac{1}{1-|z^1_i|^2}}\geq \frac{c}{2} \log\log\log\frac{1}{1-|z^1_i|}\geq s_1, \qquad i=1, \dots, l_0,
                    \]
                    and
                    \[
                        \frac{1}{4}\leq\sum_{i=1}^{l_0}(1-|z^1_i|^2)\leq \frac{1}{2}.
                    \]
                    In this construction, we will refer to the points $\Lambda_1 = \{z_1^1, z_2^1, \dots, z^1_{l_0}\}$  as the points of the first generation.
            
                Recursively, given a point $w \in \Lambda_j$ of generation $j$, we choose its successors in generation $j+1$, $z_{1, w}^{j+1}, \dots, z_{l_w, w}^{j+1}$, so that $I_{z_{i, w}^{j+1}} \subset I_w$, $I_{z_{i, w}^{j+1}} \cap I_{z_{k, w}^{j+1}} = \emptyset$ if $i \neq k$, and such that 
                    \begin{equation}
                    	\label{eqn:kahanebig}
                    	\frac{|K(z^{j+1}_{i, w})|^2}{\log\frac{1}{1-|z^{j+1}_{i, w}|^2}}\geq \frac{c}{2}\log\log\log\frac{1}{1-|z^{j+1}_{i, w}|}\geq s_{j+1}, \qquad i=1, \dots, l_w,
                    \end{equation}
                    and
                    \begin{equation}
                    \label{eqn:carlesongeometric}
                        \frac{1}{4}(1-|w|^2)\leq\sum_{i=1}^{l_w}(1-|z^{j+1}_{i, w}|^2)\leq \frac{1}{2}(1-|w|^2).
                    \end{equation}
                    Let $\Lambda_{j+1}$ be the set of all successors to points $w \in \Lambda_j$.
            
                We consider the sequence $Z = (z_n)_{n\in\N}:=\bigcup_{j=1}^\infty \Lambda_j$. Since the Carleson squares associated with points of the same generation are disjoint, $Z$ is weakly separated, and the upper bound in \eqref{eqn:carlesongeometric} implies that $\mu_Z$ satisfies the Carleson measure condition. Thus $Z$ is interpolating. On the other hand, the lower bound in \eqref{eqn:carlesongeometric} yields
                    \[
                        \sum_{z \in \Lambda_j }(1-|z|^2)\geq 4^{-j}, \qquad j \geq 1,
                    \]
                    and therefore
                    \[
                        \sum_{n\in\N} e^\frac{\lambda|K(z_n)-K(0)|^2}{-\log(1-|z_n|^2)}(1-|z_n|^2) \geq \sum_{j=1}^\infty e^{\lambda s_j}\sum_{z \in \Lambda_j}(1-|z|^2) = \infty,
                    \]
                    for every $\lambda > 0$. 
            \end{description}
        \end{proof}

\bibliographystyle{amsplain} 
\bibliography{BekolleBonamiBloch} 

\end{document}